\documentclass[12pt]{amsart}
\usepackage{amssymb,amsmath,amsthm,bm}
\usepackage[colorinlistoftodos,prependcaption,textsize=tiny]{todonotes}
\usepackage{esint}
 \definecolor{darkblue}{RGB}{0,0,160}
\usepackage{fouriernc} 
\usepackage[colorlinks=true,allcolors=darkblue]{hyperref}
\DeclareSymbolFont{usualmathcal}{OMS}{cmsy}{m}{n}
\DeclareSymbolFontAlphabet{\mathcal}{usualmathcal}

\usepackage[T1]{fontenc}

\usepackage{color}

\def\d{{\rm d}}

\def\D {{\mathcal D}}

\def \l {\langle}
\def \r {\rangle}

\def \and{\qquad\text{and}\qquad}

\newcounter{thms}

\newcounter{other}
\numberwithin{other}{section}
\newtheorem{proposition}[other]{Proposition}
\newtheorem{theorem}[thms]{Theorem}
\newtheorem*{theorem*}{Theorem}
\newtheorem*{proposition*}{Proposition}
\newtheorem{corollary}{Corollary}
\numberwithin{corollary}{thms}
\newtheorem{lemma}[other]{Lemma}
\theoremstyle{definition}
\newtheorem{remark}[other]{Remark}

\def\vint_#1{\mathchoice%
      {\mathop{\kern 0.2em\vrule width 0.6em height 0.69678ex depth -0.58065ex
              \kern -0.8em \intop}\nolimits_{\kern -0.4em#1}}%
      {\mathop{\kern 0.1em\vrule width 0.5em height 0.69678ex depth -0.60387ex
              \kern -0.6em \intop}\nolimits_{#1}}%
      {\mathop{\kern 0.1em\vrule width 0.5em height 0.69678ex depth -0.60387ex
              \kern -0.6em \intop}\nolimits_{#1}}%
      {\mathop{\kern 0.1em\vrule width 0.5em height 0.69678ex depth -0.60387ex
              \kern -0.6em \intop}\nolimits_{#1}}}
\def\vintslides_#1{\mathchoice%
      {\mathop{\kern 0.1em\vrule width 0.5em height 0.697ex depth -0.581ex
              \kern -0.6em \intop}\nolimits_{\kern -0.4em#1}}%
      {\mathop{\kern 0.1em\vrule width 0.3em height 0.697ex depth -0.604ex
              \kern -0.4em \intop}\nolimits_{#1}}%
      {\mathop{\kern 0.1em\vrule width 0.3em height 0.697ex depth -0.604ex
              \kern -0.4em \intop}\nolimits_{#1}}%
      {\mathop{\kern 0.1em\vrule width 0.3em height 0.697ex depth -0.604ex
              \kern -0.4em \intop}\nolimits_{#1}}}

\newcommand{\abs}[1]{\ensuremath{\vert#1\vert}}

\newcommand{\La}{\langle }
\newcommand{\Ra}{\rangle }

\newcommand{\aveint}[2]{\mathchoice%
      {\mathop{\kern 0.2em\vrule width 0.6em height 0.69678ex depth -0.58065ex
              \kern -0.8em \intop}\nolimits_{\kern -0.45em#1}^{#2}}%
      {\mathop{\kern 0.1em\vrule width 0.5em height 0.69678ex depth -0.60387ex
              \kern -0.6em \intop}\nolimits_{#1}^{#2}}%
      {\mathop{\kern 0.1em\vrule width 0.5em height 0.69678ex depth -0.60387ex
              \kern -0.6em \intop}\nolimits_{#1}^{#2}}%
      {\mathop{\kern 0.1em\vrule width 0.5em height 0.69678ex depth -0.60387ex
              \kern -0.6em \intop}\nolimits_{#1}^{#2}}}

\renewcommand*{\cdots}{%
  \mathinner{{\cdotp}{\cdotp}{\cdotp}}%
}

\numberwithin{equation}{section}

\title[Two-weight inequalities for multilinear commutators]{Two-weight inequalities for multilinear commutators}

 \author[Ishwari Kunwar]{Ishwari Kunwar}
 \address{\noindent School of Mathematics, Georgia Institute of Technology, \newline \indent 686 Cherry Street, Atlanta, GA 30332, USA}
 \email{ikunwar3@math.gatech.edu}

 \author[Yumeng Ou]{Yumeng Ou}
 \address{\noindent Department of Mathematics, Massachusetts Institute of Technology, \newline \indent  77 Massachusetts Avenue, Cambridge, MA 02139, USA  }
\email{yumengou@mit.edu} 

\subjclass[2010]{Primary: 42B20; Secondary: 42B25}
\keywords{Multilinear Calder\'on-Zygmund operators, multilinear dyadic operators, sparse domination, commutators, weighted BMO}

\begin{document}
\begin{abstract} 
We prove Bloom type two-weight inequalities for commutators of multilinear singular integral operators including Calder\'on-Zygmund operators and their dyadic counterparts. Such estimates are further extended to a general higher order multilinear setting. The proof involves a pointwise sparse domination of multilinear commutators.
\end{abstract}
\maketitle
 
\section{Introduction and statement of main results}

In this article, we study commutators of certain BMO functions and singular integral operators in the multilinear setting. Our basic object is 
\[
[b,T]_{\beta}(f_1,\cdots,f_m):=bT(f_1,\cdots,f_m)-T(f_1,\cdots,bf_{\beta},\cdots,f_m),
\]where $T$ is a $m$-linear operator acting on functions on $\mathbb{R}^n$, $m\geq 1$, and $b$ is understood as the pointwise multiplication operator by function $b$. 

We prove two-weight inequalities, first of their kind in the multilinear setting, for commutators of this type and its full multilinear, higher order extensions. More precisely, our goal is to understand how the operator norm of the commutator acting on weighted $L^p$ spaces is controlled by certain BMO norms (determined by the weights) of the symbol function $b$, and we are particularly interested in the cases when $T$ is a multilinear Calder\'on-Zygmund operator or its dyadic counterparts. 

Given $\vec{k}=(k_1,\ldots,k_m)$. $\forall 1\leq j\leq m$, assume $k_j\geq 0$, and let $\vec{b_j}=(b_j^1,\ldots,b_j^{k_j})$ where each $b_j^i$ is a function on $\mathbb{R}^n$. A general multilinear higher order commutator can be defined as
\begin{equation}\label{DefComm}
\begin{split}
&C^{\vec{k},m}_{\{\vec{b_j}\}}(T)\\
:=&\left[b_{m}^{k_m},\cdots,\left[b_m^1,\left[b_{m-1}^{k_{m-1}},\cdots,\left[b_{m-1}^1,\cdots,\left[b_2^1,\left[b_1^{k_1},\cdots,\left[b_1^1,T\right]_1\cdots\right]_{1}\right]_2\cdots\right]_{m-1}\cdots\right]_{m-1}\right]_m\cdots\right]_m.
\end{split}
\end{equation}Note that $C^{\vec{k},m}_{\{\vec{b_j}\}}(T)$ is also a $m$-linear operator, and it is invariant under permutations of $\left\{b_j^i\right\}_{i=1,\ldots,k_j}$, $\forall 1\leq j\leq m$. In the definition above, $k_j=0$ for some $j\in\{1,\ldots,m\}$ simply means that there is no commutator structure in the $j$-th component. It is easy to see that $[b,T]_{\beta}$ defined at the beginning is a special case of $C^{\vec{k},m}_{\{\vec{b_j}\}}(T)$.

The main results of the article are the following. We first obtain a two-weight estimate for first order multilinear commutators (Theorem \ref{basethm}), i.e. when all $k_j\leq 1$. We then develop an abstract scheme (Theorem \ref{multi2}) that bootstraps from there to the higher order case. Such a two-weight norm inequality is new even in the first order case for multilinear commutators. Moreover, lower bound estimates for multilinear commutators are also discussed (Theorem \ref{lower bound}).

Commutator estimates in terms of BMO norms of the symbol functions have attracted a great amount of interest in the past decades, as it gives straightforward characterization of BMO spaces, implies weak factorization of Hardy spaces and Div-Curl estimates, and is closely connected to Hankel operators in operator theory. We refer the reader to \cite{CRW, FS, OPS} and the references therein for more detailed discussion of these connections in the unweighted linear theory. In the linear setting (i.e. when $T$ is a linear operator), weighted and two-weight type inequalities have been recently studied in works such as \cite{CPP, HLW, HLW2, HPW}. In the multilinear setting, such estimates have also been considered where the natural class of weights and structure of the commutators are much more complicated. Lerner et al. \cite{LOPTT} proved one-weight estimates for certain first order multilinear commutators. In the work of P\'erez et al. \cite{PPTT}, one-weight estimates for iterated multilinear commutators (the full first order version of (\ref{DefComm}) with one symbol function $b_j$ appearing in each component $j$) are considered too. There is also another independent work of Tang \cite{Tang} who treated first order multilinear commutators but only for a subclass of multilinear weights: products of classical weights. Our result seems to be the first of its kind that extends such estimates to a setting where challenges from multilinear, iterated higher order, and two-weight are overcome simultaneously. 

To begin with, we introduce relevant classes of weights and BMO spaces. A positive, locally integrable function $w$ is called a Muckenhoupt $A_p$ weight if
\[
[w]_{A_p}:=\sup_{Q}\left(\fint_Q w(x) \,\d x \right)\left(\fint_Q w(x)^{1-p'}\,\d x\right)^{p-1}<\infty,\quad 1<p<\infty,
\]where the supreme is taken over all cubes $Q\subset\mathbb{R}^n$. When studying multilinear singular integrals, one usually works with weight vectors $\vec{w}=(w_1,\ldots,w_m)$ where each $w_j$ is a positive function. Let $\vec{p}=(p_1,\ldots,p_m)$ with $1<p_j<\infty$ and $1/p=1/p_1+\cdots+1/p_m$. $\vec{w}$ is said to be a \emph{multilinear $A_{\vec{p}}$ weight} if
\[
[\vec{w}]_{A_{\vec{p}}}^{1/p}:=\sup_{Q\subset \mathbb{R}^n}\left(\fint_Q \nu_{\vec{w}}\right)^{1/p} \prod_{j=1}^m\left(\fint_Q w_j^{1-p_j'}\right)^{1/p_j'}<\infty,
\]where
\[
\nu_{\vec{w}}:=\prod_{j=1}^mw_j^{p/p_j}.
\]A particular example of such weights is $\vec{w}$ with $w_j\in A_{p_j}$, $\forall j$, which is referred to as a \emph{product multiple weight}. The weight class $A_{\vec{p}}$ was first introduced in \cite{LOPTT} where the authors show that it is the correct class of weights in multilinear Calder\'on-Zygmund theory and is tied to the multilinear maximal function. In general, if $\vec{w}\in A_{\vec{p}}$, $w_j$ may not be a locally integrable function for any $j$, but instead,
\[
\vec{w}\in A_{\vec{p}}\iff
\begin{cases}
w_j^{1-p_j'}\in A_{mp_j'},\quad j=1,\ldots,m,\\
\nu_{\vec{w}}\in A_{mp}.
\end{cases}
\]

In addition to the classical BMO space, we define the following weighted BMO space associated to weight $\nu$ normed by
\[
\|b\|_{\text{BMO}(\nu)}:=\sup_{Q\subset\mathbb{R}^n}\frac{1}{\nu(Q)}\int_Q|b-\l b\r_Q|\,\d x,
\]where $\nu(Q):=\int_{Q} \nu(x)\,\d x$. Note that $\l b \r_Q$ in the above still denotes the average value of $b$ with respect to Lebesgue measure. This space was first introduced by Muckenhoupt and Wheeden \cite{MW} and Garc\'ia-Cuerva [GC] independently, and is often referred to as the \emph{Bloom BMO} space due to its role in connection with two-weight estimates for commutators studied by Bloom \cite{Bloom}. There are also dyadic versions of BMO and weighted BMO associated to certain dyadic grid $\mathcal{D}$, which are denoted as $\text{BMO}_{\mathcal{D}}$ and $\text{BMO}_{\mathcal{D}}(\nu)$. These norms are defined similarly as above but with supremum taken over only dyadic cubes.

Our first theorem provides a quantitative two-weight upper bound for the first order multilinear commutator. 
\begin{theorem}\label{basethm}
Let $\vec{p}=(p_1,\ldots,p_m)$, $1<p_j<\infty$ and $1/p=1/p_1+\cdots+1/p_m$. Let $I=\{i_1,\ldots,i_\ell\}\subset\{1,\ldots,m\}$ and $b_{i_s}\in\text{BMO}(\nu_{i_s})$, $s=1,\ldots,\ell$. For any multilinear $A_{\vec{p}}$ weights $\vec{\mu}=(\mu_1,\ldots,\mu_m)$ and $\vec{\lambda}=(\lambda_1,\ldots,\lambda_m)$ such that $\mu_{i_s}, \lambda_{i_s}\in A_{p_{i_s}}$, $\forall s=1,\ldots,\ell$, while $\mu_j=\lambda_j$, $\forall j\notin I$. Denote $\vec{q}=(p_j)_{j\notin I}$, $\vec{w}=(\mu_j)_{j\notin I}$ and assume that the $(m-\ell)$-linear weight $\vec{w}\in A_{\vec{q}}$. Then, let $\vec{\nu}=(\nu_{i_1},\ldots,\nu_{i_\ell})$ satisfy $\nu_{i_s}=(\mu_{i_s}/\lambda_{i_s})^{1/p_{i_s}}$, there holds
\[
\|[b_{i_1},\cdots,[b_{i_\ell},T]_{i_\ell}\cdots]_{i_1}(f_1,\cdots,f_m)\|_{L^p(\nu_{\vec\lambda})}\lesssim C(\vec\mu,\vec\lambda,\vec{p})\prod_{s=1}^\ell\|b_s\|_{\text{BMO}(\nu_{i_s})}\prod_{j=1}^m\|f_j\|_{L^{p_j}(\mu_j)},
\]where $T$ is a $m$-linear Calder\'on-Zygmund operator,  Haar multiplier or paraproduct (with respect to any dyadic grid), and
\[
C(\vec{\mu},\vec{\lambda},\vec{p}):=\left(\prod_{s=1}^\ell[\mu_{i_s}]_{A_{p_{i_s}}}^{\max(1,\frac{1}{p_{i_s}-1})}[\lambda_{i_s}]_{A_{p_{i_s}}}^{\frac{1}{p_{i_s}}\max(p_{i_s}, q,p'_1,\ldots,p'_m)}\right)[\vec{w}]_{A_{\vec{q}}}^{\frac{1}{q}\max(q,p'_1,\ldots,p'_m)},\quad 1/q:=\sum_{j\notin I}1/q_j.
\]
\end{theorem}
We defer the definitions of $m$-linear Calder\'on-Zygmund operators, Haar multipliers and paraproducts to Subsection \ref{SubSecDef}. The statement of the theorem above seems quite complicated, while we introduced the extra notation for the subset $I\subset \{1,\ldots,m\}$ simply to include the case of deficient commutators, i.e. when $T$ is not commutated with any symbol in some components (for example $[b_1,T]_1$). Note that in the fully degenerate case $I=\emptyset$, i.e. when there is no commutator structure at all, the theorem degenerates to the $A_2$ theorem for multilinear operators:
\[
\|T(f_1,\ldots,f_m)\|_{L^p(\nu_{\vec{w}})}\lesssim [\vec{w}]_{A_{\vec{p}}}^{\frac{1}{p}\max\left(p,p'_1,\ldots,p'_m\right)}\prod_{j=1}^m\|f_j\|_{L^{p_j}(w_j)},\quad \forall \vec{w}\in A_{\vec{p}}.
\]which for $T$ being a multilinear CZ operator was first proved by Li, Moen and Sun \cite{LMS}.

Two-weight inequalities of this type have been extensively studied in the linear setting such as for Hilbert transform \cite{Bloom, HLW}, general CZ operators \cite{HLW2, Lerner} and multi-parameter CZ operators \cite{HPW}. It is well known that weighted estimates in the multilinear setting has some intrinsic difference compared with the linear case, as there are usually quasi-Banach spaces involved and multilinear weights need not have each component being a classical linear $A_p$ weight. Our result seems to be the first in the literature to extend it to the multilinear setting and it is very interesting to know whether it is possible to weaken the assumptions on the weights in Theorem \ref{basethm}. In general, two weight estimates are known to be quite challenging, especially in the multilinear setting. Our result further magnifies the fact that commutators usually have better properties than the operator that is being commuted.

Theorem \ref{basethm} follows from a domination of the commutator by certain sparse operators (Proposition \ref{CommDomCZO}, \ref{CommDomDya}) which satisfy the desired two-weight estimate above, and the sparse domination can be obtained via a stopping time argument relying on the weak type endpoint estimate of certain maximal truncated operators. In the case of the multilinear CZ operators, such maximal truncated operator was first treated by Grafakos and Torres \cite{GrafakosTorres}, while for the dyadic operators this seems to be new. Similar technique has been used in the linear setting \cite{Lerner} to obtain an analog of the two-weight estimate for the commutator. Sparse domination of the commutator is expected to be of independent interest, as it provides a fine quantification of its boundedness which should imply not only the weighted estimates that we explore in this article, but also weak type endpoint bounds, which we plan to study in a forthcoming article.

Next, we extend the two-weight estimate to higher order commutators through an abstract two-weight bootstrapping technique, which applies to arbitrary multilinear operators, not necessarily of Calder\'on-Zygmund type. 
\begin{theorem}\label{multi2}
Let $\vec{\nu}=(\nu_1,\ldots,\nu_m)$ be a fixed multiple weight on $\mathbb{R}^n$, $\vec{p}=(p_1,\ldots,p_m)$, $1/p=1/p_1+\cdots+1/p_m$, $1< p_j<\infty$, $1\leq p<\infty$ and $\widetilde{T}$ be a $m$-linear operator satisfying
\[
\|\widetilde{T}\|_{L^{p_1}(\mu_1)\times\cdots\times L^{p_m}(\mu_m)\to L^p(\nu_{\vec\lambda})}\leq C_{n,m,\vec{p},\widetilde{T}}\left([\vec{\mu}]_{A_{\vec{p}}},[\vec{\lambda}]_{A_{\vec{p}}}\right),
\]where $C_{n,m,\vec{p},\widetilde{T}}(\cdot,\cdot)$ is an increasing function of both components, with $\vec{\mu}\in A_{\vec{p}}$, $\vec\lambda\in A_{\vec{p}}$ and $\mu_j/\lambda_j=\nu_j^{p_j}$, $\forall j$. For $\vec{k}=(k_1,\ldots,k_m)$ with $k_j\geq 0$, $\forall 1\leq j\leq m$, let $b_j^i\in \text{BMO}(\mathbb{R}^n)$, $\forall 1\leq i\leq k_j$, then there holds
\[
\left\|C^{\vec{k},m}_{\{\vec{b}\}}(\widetilde{T})\right\|_{L^{p_1}(\mu_1)\times\cdots\times L^{p_m}(\mu_m)\to L^p(\nu_{\vec\lambda})}\leq C_{n,m,\vec{p},\vec{k},\widetilde{T}}\left([\vec{\mu}]_{A_{\vec{p}}},[\vec{\lambda}]_{A_{\vec{p}}}\right)\prod_{j=1}^m\prod_{i=1}^{k_j}\|b_j^i\|_{\text{BMO}(\mathbb{R}^n)},
\]where if $k_j=0$ for some $j$, then the corresponding $\text{BMO}(\mathbb{R}^n)$ norms don't appear on the right hand side.
\end{theorem}
Combined with Theorem \ref{basethm}, this immediately implies the following two-weight upper bound estimate for higher order commutators.
\begin{corollary}\label{higherordercor}
Let $\vec{p}=(p_1,\ldots,p_m)$, $1<p_j<\infty$, $1\leq p<\infty$ and $1/p=1/p_1+\cdots+1/p_m$. Let $I=\{i_1,\ldots,i_\ell\}\subset\{1,\ldots,m\}$, $b_{i_s}^1\in\text{BMO}(\nu_{i_s})$, and $b_{i_s}^i\in\text{BMO}(\mathbb{R}^n)$, $\forall 2\leq i\leq k_{i_s}$, $s=1,\ldots,\ell$. For any multilinear $A_{\vec{p}}$ weights $\vec{\mu}=(\mu_1,\ldots,\mu_m)$ and $\vec{\lambda}=(\lambda_1,\ldots,\lambda_m)$ such that $\mu_{i_s}, \lambda_{i_s}\in A_{p_{i_s}}$, $\forall s=1,\ldots,\ell$, while $\mu_j=\lambda_j$, $\forall j\notin I$. Denote $\vec{q}=(p_j)_{j\notin I}$, $\vec{w}=(\mu_j)_{j\notin I}$ and assume that the $(m-\ell)$-linear weight $\vec{w}\in A_{\vec{q}}$. Then, let $\vec{\nu}=(\nu_{i_1},\ldots,\nu_{i_\ell})$ satisfy $\nu_{i_s}=(\mu_{i_s}/\lambda_{i_s})^{1/p_{i_s}}$, and $\vec{k}=(k_1,\ldots,k_m)$ with $k_{i_s}\geq 1$ for $s=1,\ldots,\ell$ and $k_j= 0$ for $j\notin I$. Then there holds
\[
\begin{split}
&\left\|C^{\vec{k},m}_{\{\vec{b}\}}(T)\right\|_{L^{p_1}(\mu_1)\times\cdots\times L^{p_m}(\mu_m)\to L^p(\nu_{\vec\lambda})}\\
\leq& C_{n,m,\vec{p},\vec{k},T}\left([\vec{\mu}]_{A_{\vec{p}}},[\vec{\lambda}]_{A_{\vec{p}}}\right)\left(\prod_{s=1}^\ell \|b_{i_s}^1\|_{\text{BMO}(\nu_{i_s})}\right)\left(\prod_{s=1}^\ell\prod_{i=2}^{k_{i_s}}\|b_{i_s}^i\|_{\text{BMO}(\mathbb{R}^n)}\right),
\end{split}
\]for $T$ being a $m$-linear Calder\'on-Zygmund operator, Haar multiplier or paraproduct, where if $k_{i_s}=1$ for some $s$, then the corresponding $\text{BMO}(\mathbb{R}^n)$ norms don't appear on the right hand side.
\end{corollary}

Note that the choice $\left\{b_{i_s}^1\right\}$ above doesn't play any particular role as the commutator is invariant under permutations of $\left\{b_j^i\right\}$ for each fixed $j$. Theorem \ref{multi2} is proven using a method involving the Cauchy integral formula, which goes back to the work of Coifman, Rochberg and Weiss \cite{CRW}. In the linear, one-weight case when all functions $b_j$ are the same, this result was proved by Chung, Pereyra and P\'erez in \cite{CPP}. In the two-weight case, the only previously known result along this line of research is by Hyt\"onen \cite{Hy}, where he applies the Cauchy integral method to obtain a linear version of Theorem \ref{multi2} when $b_j=b\in \text{BMO}\cap\text{BMO}(\nu)$, $\forall j$. Hyt\"onen's result was first proved by Holmes and Wick \cite{HW}, where a different method involving decomposition into dyadic shift operators is applied. In the multilinear case, much less is known. The Cauchy integral formula (in a simple form) was first used by P\'erez and Torres \cite{PerezTorres} to study certain first order multilinear commutators, and it was further extended to a very general setting by B\'enyi et al. \cite{BMMST} where different types of multilinear operators (not necessarily CZ), weights and BMO spaces are considered. Our result continues the line of research and seems to be the first attempt to extend the Cauchy integral method to the multilinear two-weight setting.

Furthermore, we obtain a lower bound estimate for the commutator, which provides a characterization of BMO via multilinear commutators. 
\begin{theorem}\label{lower bound}
The following statements are equivalent.
\begin{enumerate}
\item $b \in \text{BMO}(\mathbb{R}^n)$.
\item For all $1/p=1/p_1+\cdots+1/p_m$ with $1 < p_j< \infty$, $\beta \in \{1, \ldots, m \}$, and weight $\vec{w} = (w_1, \ldots, w_m)$ with $w_\beta\in A_{p_\beta}$ and $(m-1)$-linear weight $(w_1,\ldots,\widehat{w_\beta},\ldots,w_m)\in A_{\vec{q}}$, $\vec{q}=(p_1,\ldots,\widehat{p_\beta},\ldots,p_m)$, the following map is bounded:
\[[b, T]_{\beta}: L^{p_1}(w_1) \times \cdots \times L^{p_m}(w_m) \rightarrow L^p(\nu_{\vec{w}}),\]where $T$ is any $m$-linear Calder\'on-Zygmund operator, Haar multiplier or paraproduct.
\item For all dyadic grid $\mathcal{D}$, there exist some choices of $\vec{p}$, $\beta$ as above, some $\vec{w}$ with $w_j\in A_{p_j}$, $\forall j$, and some $\bm{\alpha}$, so that the following map is bounded:
\[[b, P_{\bm{\epsilon}}^{\bm{\alpha}}]_{\beta}: L^{p_1}(w_1) \times \cdots \times L^{p_m}(w_m) \rightarrow L^p(\nu_{\vec{w}}),
\]where $P_{\bm{\epsilon}}^{\bm{\alpha}}$ is any Haar multiplier defined in (\ref{DefHaarM}) with respect to $\mathcal{D}$ satisfying that $\{|\epsilon_I|\}_I$ are bounded from below uniformly and $\alpha_j\neq 1$ for some $j\in\{1,\ldots,\widehat{\beta},\ldots,m\}$.

\end{enumerate}
\end{theorem}
This result shows that multilinear dyadic commutator is a representative class for commutators of other multilinear operators: given a function $b$, if for each dyadic grid the commutator $[b,P_{\bm{\epsilon}}^{\bm{\alpha}}]_{\beta}$ is bounded on weighted $L^{\vec{p}}$ for \emph{some} choices of Haar multiplier $P_{\bm{\epsilon}}^{\bm{\alpha}}$, parameter $\beta$, and multilinear weight, then commutator $[b,T]_{\beta'}$ is bounded on \emph{any} weighted $L^{\vec{q}}$ for \emph{any} choices of continuous or dyadic multilinear CZ operator $T$, parameter $\beta'$, and multilinear weight.

In the linear setting, dyadic operators have been extensively studied as model cases and tools to handle various problems in the continuous setting, which is why we expect our results above concerning dyadic operators (Haar multipliers and paraproducts) to have some further applications in the advancement of the multilinear theory. For example, in a very recent work of the second author with Li, Martikainen and Vuorinen \cite{LMOV}, a bilinear representation theorem is proved which enables one to represent bilinear CZ operators as averages of bilinear dyadic shifts (higher complexity version of Haar multipliers) and paraproducts. It would be interesting to know whether similar versions of Theorem \ref{basethm} and \ref{lower bound} hold for dyadic shifts as well. It is also worth noticing that multilinear Haar multipliers and paraproducts arise naturally in the decomposition of $m$-fold pointwise product of functions $f_1,\ldots,f_m$, which we refer to \cite{Kunwar} for more details.

\begin{remark}
In Theorem \ref{basethm}, \ref{multi2}, \ref{lower bound} above, if one is only interested in studying the dyadic operators with respect to dyadic grid $\mathcal{D}$, one can replace the BMO spaces that appear by their dyadic versions and the same results remain valid.
\end{remark}

The article is organized as follows. In Section \ref{SecSparse}, we prove a sparse bound for first order multilinear CZ commutators, which, combined with Subsection \ref{SubSecFirst}, will complete the proof of Theorem \ref{basethm} for CZ operators. In the rest of Section \ref{SecBloom}, Theorem \ref{multi2} and \ref{lower bound} will be demonstrated. We then discuss the case with dyadic operators (Haar multipliers and paraproducts) of Theorem \ref{basethm} in Appendix \ref{SecApp}, as it proceeds almost parallel to the continuous one.

Before finishing the Introduction, we state the definitions of the multilinear operators that appear in the main theorems.
\subsection{Definitions of multilinear CZ and dyadic operators}\label{SubSecDef}
Let $T$ be a $m$-linear operator mapping $(\mathcal{S}(\mathbb{R}^n))^m$ into $\mathcal{S}'(\mathbb{R}^n)$. $T$ is called a \emph{$m$-linear Calder\'on-Zygmund operator} if for appropriate functions $\{f_1,\ldots,f_m\}$,
\[
T(f_1,\ldots,f_m)(x)=\int_{(\mathbb{R}^n)^m} K(x,y_1,\ldots,y_m)f_1(y_1)\cdots f_m(y_m)\,dy_1\cdots dy_m,
\]whenever $x\notin\cap_{j=1}^m\text{supp }f_j$, where the kernel $K$ is assumed to be locally integrable away from the diagonal $x=y_1=\cdots =y_m$ in $(\mathbb{R}^n)^{m+1}$, satisfying for some $C,\epsilon>0$ the \emph{size estimate}
\[
|K(y_0,y_1,\ldots,y_m)|\leq\frac{C}{(\sum_{k,\ell=0}^m|y_k-y_{\ell}|)^{nm}}
\]for all $(y_0,y_1,\ldots,y_m)\in (\mathbb{R}^n)^{m+1}$ away from the diagonal, and the smoothness estimate
\[
|K(y_0,\ldots,y_j,\ldots,y_m)-K(y_0,\ldots,y_j',\ldots,y_m)|\leq \frac{C|y_j-y_j'|^\epsilon}{(\sum_{k,\ell=0}^m|y_k-y_{\ell}|)^{nm+\epsilon}}
\]whenever $0\leq j\leq m$ and $|y_j-y_j'|\leq \frac12 \max_{0\leq k\leq m}|y_j-y_k|$. We also make the a priori assumption that $T$ maps $L^{q_1}\times\cdots\times L^{q_m}$ boundedly into $L^{q}$ for some $1<q_j<\infty$ satisfying $\sum_{j=1}^m 1/q_j=1/q$. We refer the reader to \cite{GrafakosTorres, LernerNazarov, CondeRey, Li} and the references therein for more details and properties of multilinear CZ operators. Note that our results remain valid for more general multilinear CZ operators with weaker kernel assumptions, for example kernels satisfying Dini type estimates or $m$-linear $L^r$-H\"ormander condition as studied in \cite{Li}, as what matters here (more precisely, in the proof of Proposition \ref{CommDomCZO} of Section \ref{SecSparse}) is the boundedness of the corresponding maximal truncated operators. 

Next, we define the multilinear dyadic operators that appear in Theorem \ref{basethm}, \ref{lower bound}. Let $\mathcal{D}$ be a fixed dyadic grid on $\mathbb{R}^n$. For all $I\in\mathcal{D}$, $\{h_I^{\alpha_j}\}$ are the cancellative Haar functions for $\{0,1\}^n\ni\alpha_j\neq\vec{1}$, and $h_I^1:=h_I^{\vec{1}}:=|I|^{-1/2}\chi_I$ is the non-cancellative Haar function. First, define the multilinear Haar multiplier associated to the bounded sequence $\bm{\epsilon}=\{\epsilon_I\}_{I\in\mathcal{D}}$ as
\begin{equation}\label{DefHaarM}
P_{\bm{\epsilon}}^{\bm{\alpha}}(\vec{f}):=\sum_{I\in\mathcal{D}}\epsilon_I\La f_1,h_I^{\alpha_1}\Ra \cdots \La f_m,h_I^{\alpha_m}\Ra h_I^{\alpha_{m+1}}|I|^{-(m-1)/2},
\end{equation}where at least two of all $\{\alpha_1,\ldots,\alpha_{m+1}\}$ are not equal to $\vec{1}$, i.e. there are at least two out of the $m+1$ Haar functions that are cancellative. Then, for $g\in\text{BMO}_{\mathcal{D}}$, define the multilinear dyadic paraproduct associated to the bounded sequence $\bm{\epsilon}=\{\epsilon_I\}_{I\in\mathcal{D}}$ as
\begin{equation}\label{DefPara}
\pi_{g,\bm{\epsilon}}^{\bm{\alpha}}(\vec{f}):=\sum_{I\in\mathcal{D}}\epsilon_I\La g, h_I^{\alpha_1} \Ra\left(\prod_{j=1}^m \La f_j, h_I^{\alpha_{j+1}}\Ra\right) h_I^{\alpha_{m+2}}|I|^{-m/2},
\end{equation}where $\alpha_1\neq \vec{1}$ and at least one of the superscripts $\{\alpha_2,\ldots,\alpha_{m+2}\}$ is not equal to $\vec{1}$.

\section{Sparse domination of multilinear commutators}\label{SecSparse}
In this section, we present the first step (Proposition \ref{CommDomCZO} below) of the proof of Theorem \ref{basethm} for multilinear Calder\'on-Zygmund operators, which reduces it to the estimates of certain sparse operators. A similar version of this reduction for the dyadic operators is given in Appendix \ref{SecApp} (Proposition \ref{CommDomDya}). A collection $\mathcal{S}$ of cubes in $\mathbb{R}^n$ is called \emph{$\eta$-sparse} if there exists $E_I\subset I$ for all $I\in\mathcal{S}$ such that $|E_I|>\eta|I|$ and $\{E_I\}_{I\in\mathcal{S}}$ pairwise disjoint.
\begin{proposition}\label{CommDomCZO}
Let $T$ be a $m$-linear Calder\'on-Zygmund operator and $I=\{i_1,\ldots,i_\ell\}\subset\{1,\ldots,m\}$. Given locally integrable functions $\bm{b}=(b_{i_1},\ldots,b_{i_\ell})$ on $\mathbb{R}^n$, there exists a constant $C=C(n,T)$ so that for any bounded functions $\bm{f}=(f_1,\ldots,f_m)$ with compact support, there exists $3^n$ sparse collections $\mathcal{S}_j=\mathcal{S}_j(T,\bm{f},\bm{b})$ of dyadic cubes, $j=1,\ldots, 3^n$ such that 
\begin{equation}\label{sparseCZO}
\left|[b_{i_1},\cdots,[b_{i_\ell},T]_{i_\ell}\cdots]_{i_1}(\bm{f})\right|\leq C\left(\sum_{j=1}^{3^n}\sum_{\vec{\gamma}\in \{1,2\}^\ell}\mathcal{A}^{\vec{\gamma}}_{\mathcal{S}_j,\bm{b}}(\bm{f})\right),\quad \text{a.e.}
\end{equation}where 
\[
\mathcal{A}^{\vec{\gamma}}_{\mathcal{S}_j,\bm{b}}(\bm{f}):=\sum_{Q\in\mathcal{S}_j}\left(\prod_{s=1}^\ell\Gamma(b_{i_s},f_{i_s},Q,\gamma_{i_s})\right)\left(\prod_{j\notin I}\langle |f_j|\rangle_Q\right)\chi_Q,
\]
\[
\Gamma(b,f,Q,\gamma):=\begin{cases} \left|b-\langle b\rangle_Q\right|\left\langle |f|\right\rangle_Q &\text{if}\,\gamma=1,\\ \left\langle \left|(b-\langle b\rangle_Q)f\right|\right\rangle_Q &\text{if}\, \gamma=2. \end{cases}
\]
\end{proposition}

\begin{proof}
For the sake of brevity, we prove the proposition only in the bilinear setting and assume that the commutator is full ($I$=\{1,\ldots,m\}). It will be easy to see that the argument extends to the multilinear setting in the obvious way, and the deficient commutator case is even easier to treat. Let $m=2$, then (\ref{sparseCZO}) reduces to the domination  
\[
|[b_2,[b_1,T]_1]_2(f_1,f_2)|\leq C\sum_{j=1}^{3^n}\left(\mathcal{A}^{(1,1)}_{\mathcal{S}_j,\bm{b}}(\bm{f})+\mathcal{A}^{(1,2)}_{\mathcal{S}_j,\bm{b}}(\bm{f})+\mathcal{A}^{(2,1)}_{\mathcal{S}_j,\bm{b}}(\bm{f})+\mathcal{A}^{(2,2)}_{\mathcal{S}_j,\bm{b}}(\bm{f})\right)\quad\text{a.e.}
\]for some choices of sparse collections $\mathcal{S}_j$. We claim that it suffices to show for any fixed cube $Q_0\subset\mathbb{R}^n$ that there exists a $\frac{1}{2}$-sparse collection $\mathcal{S}\subset\mathcal{D}(Q_0)$ such that for a.e. $x\in Q_0$,
\begin{equation}\label{reduceQ0}
\left|[b_2,[b_1,T]_1]_2(f_1\chi_{3Q_0},f_2\chi_{3Q_0})\right|\leq C\sum_{Q\in\mathcal{S}}\left(\sum_{\vec{\gamma}\in\{1,2\}^2}\widetilde{\Gamma}(b_1,f_1,Q,\gamma_1)\widetilde{\Gamma}(b_2,f_2,Q,\gamma_2)\right)\chi_Q,
\end{equation}where
\[
\widetilde{\Gamma}(b_i,f_i,Q,\gamma_i):=\begin{cases} \left|b_i-\langle b_i\rangle_{R_Q}\right|\left\langle |f_i|\right\rangle_{3Q} &\text{if}\,\gamma_i=1,\\ \left\langle \left|(b_i-\langle b_i\rangle_{R_Q})f_i\right|\right\rangle_{3Q} &\text{if}\, \gamma_i=2, \end{cases}
\]and $R_Q$ is a cube from one of the fixed $3^n$ dyadic grids such that $3Q\subset R_Q$ and $|R_Q|\leq 9^n|Q|$. The reduction to estimate (\ref{reduceQ0}) is fairly standard, which we omit and refer to \cite{Lerner} for instance for a justification.

Estimate (\ref{reduceQ0}) will follow from iterating the following claim: there exists a disjoint collection of cubes $P_j\in\mathcal{D}(Q_0)$ such that $\sum_j|P_j|<\frac{1}{2}|Q_0|$ and for a.e. $x\in Q_0$ there holds
\begin{equation}\label{iteration1}
\begin{split}
\left|[b_2,[b_1,T]_1]_2(f_1\chi_{3Q_0},f_2\chi_{3Q_0})(x)\right|\leq &C\left(\sum_{\vec{\gamma}\in\{1,2\}^2}\widetilde{\Gamma}(b_1,f_1,Q_0,\gamma_1)\widetilde{\Gamma}(b_2,f_2,Q_0,\gamma_2)\right)\\
&+\sum_j\left|[b_2,[b_1,T]_1]_2(f_1\chi_{3P_j},f_2\chi_{3P_j})(x)\right|\chi_{P_j}(x).
\end{split}
\end{equation}By the disjointness of $\{P_j\}$ (which will be constructed later), (\ref{iteration1}) can be deduced from the tail estimate
\begin{equation}\label{iteration2}
\begin{split}
&\left|[b_2,[b_1,T]_1]_2(f_1\chi_{3Q_0},f_2\chi_{3Q_0})\right|\chi_{Q_0\setminus\bigcup_j P_j}\\
&\quad+\sum_j\left|[b_2,[b_1,T]_1]_2(f_1\chi_{3Q_0},f_2\chi_{3Q_0})-[b_2,[b_1,T]_1]_2(f_1\chi_{3P_j},f_2\chi_{3P_j})\right|\chi_{P_j}\\
\leq & C\left(\sum_{\vec{\gamma}\in\{1,2\}^2}\widetilde{\Gamma}(b_1,f_1,Q_0,\gamma_1)\widetilde{\Gamma}(b_2,f_2,Q_0,\gamma_2)\right).
\end{split}
\end{equation}

To see (\ref{iteration2}), we consider the following multilinear maximal truncated operator
\[
\mathcal{M}_{T,Q_0}(f_1,f_2)(x):=\sup_{Q:\,x\in Q\subset Q_0}\underset{\xi\in Q}{\text{esssup}}\left|T(f_1\chi_{3Q_0},f_2\chi_{3Q_0})(\xi)-T(f_1\chi_{3Q},f_2\chi_{3Q})(\xi)\right|.
\]It is proven in \cite{Li} that $\mathcal{M}_{T,Q_0}: L^1\times L^1\to L^{1/2,\infty}$ and for a.e. $x\in Q_0$,
\begin{equation}\label{maxtrunc}
\left|T(f_1\chi_{3Q_0},f_2\chi_{3Q_0})(x)\right|\leq C_n\|T\|_{L^1\times L^1\to L^{1/2,\infty}}|f_1(x)f_2(x)|+\mathcal{M}_{T,Q_0}(f_1,f_2)(x).
\end{equation}

Using the fact that $[b_i,T]_i=[b_i-c,T]_i$ for any constant $c$, one can unravel the commutator to bound the LHS of (\ref{iteration2}) by
\[
A_1+A_2+B_1+B_2+C_1+C_2+D_1+D_2
\]where
\[
\begin{split}
&A_1:=\left|b_1-\langle b_1\rangle_{R_{Q_0}}\right|\left|b_2-\langle b_2\rangle_{R_{Q_0}}\right|\left|T\left(f_1\chi_{3Q_0},f_2\chi_{3Q_0}\right)\right|\chi_{Q_0\setminus\bigcup_j P_j},\\
&A_2:=\left|b_1-\langle b_1\rangle_{R_{Q_0}}\right|\left|b_2-\langle b_2\rangle_{R_{Q_0}}\right|\sum_j\left|T\left(f_1\chi_{3Q_0},f_2\chi_{3Q_0}\right)-T\left(f_1\chi_{3P_j},f_2\chi_{3P_j}\right)\right|\chi_{P_j},\\
&B_1:=\left|b_2-\langle b_2\rangle_{R_{Q_0}}\right|\left|T\left((b_1-\langle b_1\rangle_{R_{Q_0}})f_1\chi_{3Q_0},f_2\chi_{3Q_0}\right)\right|\chi_{Q_0\setminus\bigcup_j P_j},\\
&B_2:=\left|b_2-\langle b_2\rangle_{R_{Q_0}}\right|\sum_j\left|T\left((b_1-\langle b_1\rangle_{R_{Q_0}})f_1\chi_{3Q_0},f_2\chi_{3Q_0}\right)-T\left((b_1-\langle b_1\rangle_{R_{Q_0}})f_1\chi_{3P_j},f_2\chi_{3P_j}\right)\right|\chi_{P_j},\\
&C_1:=\left|b_1-\langle b_1\rangle_{R_{Q_0}}\right|\left|T\left(f_1\chi_{3Q_0},(b_2-\langle b_2\rangle_{R_{Q_0}})f_2\chi_{3Q_0}\right)\right|\chi_{Q_0\setminus\bigcup_j P_j},\\
&C_2:=\left|b_1-\langle b_1\rangle_{R_{Q_0}}\right|\sum_j\left|T\left(f_1\chi_{3Q_0},(b_2-\langle b_2\rangle_{R_{Q_0}})f_2\chi_{3Q_0}\right)-T\left(f_1\chi_{3P_j},(b_2-\langle b_2\rangle_{R_{Q_0}})f_2\chi_{3P_j}\right)\right|\chi_{P_j},\\
&D_1:=\left|T\left((b_1-\langle b_1\rangle_{R_{Q_0}})f_1\chi_{3Q_0},(b_2-\langle b_2\rangle_{R_{Q_0}})f_2\chi_{3Q_0}\right)\right|\chi_{Q_0\setminus\bigcup_j P_j},\\
&D_2:=\sum_j\Big|T\left((b_1-\langle b_1\rangle_{R_{Q_0}})f_1\chi_{3Q_0},(b_2-\langle b_2\rangle_{R_{Q_0}})f_2\chi_{3Q_0}\right)\\
&\qquad\qquad\qquad-T\left((b_1-\langle b_1\rangle_{R_{Q_0}})f_1\chi_{3P_j},(b_2-\langle b_2\rangle_{R_{Q_0}})f_2\chi_{3P_j}\right)\Big|\chi_{P_j}.
\end{split}
\]

We now define the exceptional set $E=\bigcup_{j=1}^4 E_j$ where
\[
E_1:=\left\{x\in Q_0:\,\max\left(|f_1f_2|(x),\mathcal{M}_{T,Q_0}(f_1,f_2)(x)\right)>C\langle |f_1|\rangle_{3Q_0}\langle |f_2|\rangle_{3Q_0}\right\},
\]
\[
\begin{split}
E_2:=\Big\{x\in Q_0:\,&\max\Big(\left|\left(b_1-\langle b_1\rangle_{R_{Q_0}}\right)f_1f_2\right|(x),\\
&\mathcal{M}_{T,Q_0}\left(\big(b_1-\langle b_1\rangle_{R_{Q_0}}\big)f_1,f_2\right)(x)\Big)>C\left\langle\left|\big(b_1-\langle b_1\rangle_{R_{Q_0}}\big)f_1\right|\right\rangle_{3Q_0}\langle |f_2|\rangle_{3Q_0}\Big\},
\end{split}
\]
\[
\begin{split}
E_3:=\Big\{x\in Q_0:\,&\max\Big(\left|f_1\left(b_2-\langle b_2\rangle_{R_{Q_0}}\right)f_2\right|(x),\\
&\mathcal{M}_{T,Q_0}\left(f_1,\big(b_2-\langle b_2\rangle_{R_{Q_0}}\big)f_2\right)(x)\Big)>C\langle|f_1|\rangle_{3Q_0}\left\langle\left|\big(b_2-\langle b_2\rangle_{R_{Q_0}}\big)f_2\right|\right\rangle_{3Q_0} \Big\},
\end{split}
\]
\[
\begin{split}
E_4:=\Big\{x\in Q_0:\,&\max\Big(\left|\left(b_1-\langle b_1\rangle_{R_{Q_0}}\right)\left(b_2-\langle b_2\rangle_{R_{Q_0}}\right)f_1f_2\right|(x),\\
&\qquad\qquad\mathcal{M}_{T,Q_0}\left(\big(b_1-\langle b_1\rangle_{R_{Q_0}}\big)f_1,\big(b_2-\langle b_2\rangle_{R_{Q_0}}\big)f_2\right)(x)\Big)\\
&>C\left\langle\left|\big(b_1-\langle b_1\rangle_{R_{Q_0}}\big)f_1\right|\right\rangle_{3Q_0}\left\langle\left|\big(b_2-\langle b_2\rangle_{R_{Q_0}}\big)f_2\right|\right\rangle_{3Q_0} \Big\}.
\end{split}
\]For $C$ chosen sufficiently large, there holds $|E|\leq\frac{1}{2^{n+2}}|Q_0|$. The collection $\{P_j\}$ are constructed by the stopping cubes obtained via the Calder\'on-Zygmund decomposition of function $\chi_E$ at level $\lambda=\frac{1}{2^{n+1}}$. In particular, there holds for each $j$ that $\frac{1}{2^{n+1}}|P_j|\leq |P_j\cap E|\leq \frac{1}{2}|P_j|$ and $|E\setminus \bigcup_j P_j|=0$. Hence, $\sum_j|P_j|<\frac{1}{2}|Q_0|$ and $P_j\cap E^c\neq\emptyset$.

It's thus left to show that the terms $A_1,\ldots, D_2$ are bounded by the RHS of (\ref{iteration2}). Taking $A_1, A_2$ as examples, if $x\in Q_0\setminus\bigcup_j P_j$, then $x\notin E_1$, hence
\[
A_1(x)\leq C\widetilde{\Gamma}(b_1,f_1,Q_0,1)\widetilde{\Gamma}(b_2,f_2,Q_0,1)
\]implied by property (\ref{maxtrunc}). If $x\in P_j$, by definition of $\mathcal{M}_{T,Q_0}$, 
\[
\begin{split}
A_2(x)\leq&\left|b_1-\langle b_1\rangle_{R_{Q_0}}\right|\left|b_2-\langle b_2\rangle_{R_{Q_0}}\right|\underset{\xi\in P_j}{\text{esssup}}\left|T\left(f_1\chi_{3Q_0},f_2\chi_{3Q_0}\right)-T\left(f_1\chi_{3P_j},f_2\chi_{3P_j}\right)\right|(\xi)\\
\leq &C\widetilde{\Gamma}(b_1,f_1,Q_0,1)\widetilde{\Gamma}(b_2,f_2,Q_0,1).
\end{split}
\]The rest of the terms can be estimated similarly, thus the proof is complete.
\end{proof}

\section{Bloom inequalities for multilinear commutators}\label{SecBloom}




\subsection{First order multilinear commutators}\label{SubSecFirst}


According to Proposition \ref{CommDomCZO} and \ref{CommDomDya}, Theorem \ref{basethm} would follow from the two-weight inequality below for the multilinear sparse operator adapted to the symbol functions $\bm{b}$ of the commutator.
\begin{lemma}\label{SparseBloom}
Let $\vec{p}=(p_1,\ldots,p_m)$ with $1<p_1,\ldots,p_m<\infty$ and $1/p=\sum_{j=1}^m1/p_j$. For any $I=\{i_1,\ldots,i_\ell\}\subset\{1,\ldots,m\}$, let $\bm{b}=(b_{i_1},\ldots,b_{i_\ell})$ with $b_{i_s}\in\text{BMO}(\nu_{i_s})$, $s=1,\ldots,\ell$. Given multilinear $A_{\vec{p}}$ weights $\vec{\mu}=(\mu_1,\ldots,\mu_m)$ and $\vec{\lambda}=(\lambda_1,\ldots,\lambda_m)$ such that $\mu_{i_s},\lambda_{i_s}\in A_{p_j}$, $\forall s=1,\ldots,\ell$ while $\mu_j=\lambda_j$, $\forall j\notin I$. Denote $\vec{q}=(p_j)_{j\notin I}$, $\vec{w}=(\mu_j)_{j\notin I}$ and assume that the $(m-\ell)$-linear weight $\vec{w}\in A_{\vec{q}}$. Let $\vec{\nu}=(\nu_{i_1},\ldots,\nu_{i_\ell})$ satisfy $\nu_{i_s}=(\mu_{i_s}/\lambda_{i_s})^{1/p_{i_s}}$. Then for any sparse collection $\mathcal{S}$, and $\vec{\gamma}\in\{1,2\}^\ell$, there holds
\[
\|\mathcal{A}^{\vec{\gamma}}_{\mathcal{S},\bm{b}}(\bm{f})\|_{L^p(\nu_{\vec\lambda})}\lesssim C(\vec{\mu},\vec{\lambda},\vec{p}) \prod_{s=1}^\ell\|b_{i_s}\|_{\text{BMO}(\nu_{i_s})}\prod_{j=1}^m\|f_j\|_{L^{p_j}(\mu_j)},
\]where
\[
C(\vec{\mu},\vec{\lambda},\vec{p})=\left(\prod_{s=1}^\ell[\mu_{i_s}]_{A_{p_{i_s}}}^{\max(1,\frac{1}{p_{i_s}-1})}[\lambda_{i_s}]_{A_{p_{i_s}}}^{\frac{1}{p_{i_s}}\max(p_{i_s}, q,p'_1,\ldots,p'_m)}\right)[\vec{w}]_{A_{\vec{q}}}^{\frac{1}{q}\max(q,p'_1,\ldots,p'_m)},\quad 1/q:=\sum_{j\notin I}1/q_j.
\]
\end{lemma}

\begin{proof}
Let $0\leq\ell_1\leq\ell_2\leq m$ and consider the sparse operator
\[
\mathcal{A}(\bm{f})=\sum_{Q\in\mathcal{S}}\left(\prod_{s=1}^{\ell_1}|b_s-\langle b_s\rangle_Q|\langle |f_s|\rangle_Q\right)\left(\prod_{s=\ell_1+1}^{\ell_2}\left\langle \left|\left(b_s-\langle b_s\rangle_Q\right)f_s\right|\right\rangle_Q\right)\left(\prod_{s=\ell_2+1}^m\langle |f_s|\rangle_Q\right)\chi_Q,
\]where we have omitted the dependence on $\mathcal{S}$ and $\bm{b}$ for the sake of brevity. By symmetry, it suffices to prove the lemma for $\mathcal{A}$ (with $\ell$=$\ell_2$) and we assume $0<\ell_1<\ell_2<m$ as this is the most difficult case.

According to Lemma 5.1 in \cite{Lerner}, there exists a sparse collection $\widetilde{\mathcal{S}}$ of dyadic cubes such that $\mathcal{S} \subset \widetilde{\mathcal{S}},$ and for a.e. $x \in Q \in \mathcal{S},$ $$ \left| b(x) - \La b \Ra_Q\right| \leq C_1 \sum_{J\in \widetilde{\mathcal{S}}:\,J\subseteq Q} \left\La \left| b - \La b \Ra_J \right| \right\Ra_J \chi_J.$$ 
Applying this result iteratively for $(\ell_2-\ell_1)$ times, one can find a sparse collection (still denoted as $\widetilde{S}$) such that $\mathcal{S}\subset\widetilde{\mathcal{S}}$ and for a.e. $x\in Q\in\mathcal{S}$,
\[
\left| b_s(x) - \La b_s \Ra_Q\right| \leq C_1 \sum_{J\in \widetilde{\mathcal{S}}:\,J\subseteq Q} \left\La \left| b_s - \La b_s \Ra_J \right| \right\Ra_J \chi_J,\quad \forall \ell_1+1\leq s\leq\ell_2.
\]
Therefore, for $\ell_1+1\leq s\leq \ell_2$,
\begin{eqnarray*}
\left\La \left| \left(b_s - \La b_s \Ra_Q\right) f_s \right| \right\Ra_Q
&\lesssim&\frac{1}{|Q|} \int_Q \sum_{J\in \widetilde{\mathcal{S}}, J\subseteq Q} \La \left| b_s - \La b_s \Ra_J \right| \Ra_J |f_s| \chi_J\\
&\leq& \frac{1}{|Q|} \left\Vert b_s\right\Vert_{\text{BMO}(\nu_s)}\sum_{J\in \widetilde{\mathcal{S}}, J\subseteq Q} \La |f_s| \Ra_J \nu_s(J)\\
&=& \left\Vert b_s\right\Vert_{\text{BMO}(\nu_s)}\left\La \mathcal{A}^1_{\widetilde{\mathcal{S}}}(f_s) \nu_s \right\Ra_Q,
\end{eqnarray*}
where $\mathcal{A}^1_{\widetilde{\mathcal{S}}}(f):=\sum_{Q\in\widetilde{\mathcal{S}}}\langle |f|\rangle_Q\chi_Q$ is the classical linear sparse operator. This implies that
\[
\mathcal{A}(\bm{f})\lesssim \sum_{Q\in\mathcal{S}}\left(\prod_{s=1}^{\ell_1}|b_s-\langle b_s\rangle_Q|\langle |f_s|\rangle_Q\right)\left(\prod_{s=\ell_1+1}^{\ell_2}\|b_s\|_{\text{BMO}(\nu_s)}\left\langle \mathcal{A}^1_{\widetilde{\mathcal{S}}}(f_s)\nu_s\right\rangle_Q\right)\left(\prod_{s=\ell_2+1}^m\langle |f_s|\rangle_Q\right)\chi_Q.
\]

Next, for any $1\leq s\leq\ell_1$ and $Q\in\mathcal{S}$, the trivial estimate
\[
|b_s-\langle b_s\rangle_Q|\langle |f_s|\rangle_Q\chi_Q\leq \sum_{J\in\mathcal{S}}|b_s-\langle b_s\rangle_J|\langle |f_s|\rangle_J\chi_J
\]implies that $\mathcal{A}(\bm{f})$ admits the further domination
\[
\begin{split}
\mathcal{A}(\bm{f})\lesssim&\left(\prod_{s=1}^{\ell_1}\left(\sum_{J\in\mathcal{S}}|b_s-\langle b_s\rangle_J|\langle |f_s|\rangle_J\chi_J\right)\right)\sum_{Q\in\mathcal{S}}\left(\prod_{s=\ell_1+1}^{\ell_2}\|b_s\|_{\text{BMO}(\nu_s)}\left\langle \mathcal{A}^1_{\widetilde{\mathcal{S}}}(f_s)\nu_s\right\rangle_Q\right)\left(\prod_{s=\ell_2+1}^m\langle |f_s|\rangle_Q\right)\chi_Q\\
=: &\left(\prod_{s=1}^{\ell_1}\mathcal{T}_{b_s}(f_s)\right)\mathcal{A}^{m-\ell_1}_{\mathcal{S}}\left(\mathcal{A}^1_{\widetilde{\mathcal{S}}}(f_{\ell_1+1})\nu_{\ell_1+1},\ldots,\mathcal{A}^1_{\widetilde{\mathcal{S}}}(f_{\ell_2})\nu_{\ell_2},f_{\ell_2+1},\ldots,f_m\right)\left(\prod_{s=\ell_1+1}^{\ell_2}\|b_s\|_{\text{BMO}(\nu_s)}\right)\\
=: &\left(\prod_{s=1}^{\ell_1}\mathcal{T}_{b_s}(f_s)\right) \left(\prod_{s=\ell_1+1}^{\ell_2}\|b_s\|_{\text{BMO}(\nu_s)}\right) \mathcal{B},
\end{split}
\]where $\mathcal{A}^{m-\ell_1}_{\mathcal{S}}$ denotes the classical $(m-\ell_1)$-linear sparse operator.

Fix $1<p_1,\ldots,p_m<\infty$, $1/p=1/p_1+\cdots+1/p_m$. Let weights $\vec{\lambda}=(\lambda_1,\ldots,\lambda_{\ell_2},\mu_{\ell_2+1},\ldots,\mu_m)$, $\vec{\mu}=(\mu_1,\ldots,\mu_m)$ be such that $\lambda_s,\mu_s\in A_{p_s}$, $s=1,\ldots,\ell_2$ and $\vec{w}=(\mu_{\ell_2+1},\ldots,\mu_m)\in A_{\vec{q}}$ is a $(m-\ell_2)$-linear weight with $\vec{q}=(p_{\ell_2+1},\ldots,p_m)$. We also recall the notation $\nu_s=(\mu_s/\lambda_s)^{1/p_s}$, $s=1,\ldots,\ell_2$. Then by H\"older's inequality,
\[
\begin{split}
&\|\mathcal{A}(\bm{f})\|_{L^p(\nu_{\vec\lambda})}\\
\lesssim&\left(\prod_{s=\ell_1+1}^{\ell_2}\|b_s\|_{\text{BMO}(\nu_s)}\right)\left(\int \left|\prod_{s=1}^{\ell_1}\mathcal{T}_{b_s}(f_s)\lambda_s^{1/p_s}\right|^p\left|\mathcal{B}\right|^p\left(\prod_{s=\ell_1+1}^{\ell_2}\lambda_s^{p/p_s}\right)\left(\prod_{s=\ell_2+1}^m\mu_s^{p/p_s}\right)\,\d x\right)^{1/p}\\
\leq &\left(\prod_{s=\ell_1+1}^{\ell_2}\|b_s\|_{\text{BMO}(\nu_s)}\right)\prod_{s=1}^{\ell_1}\left(\int \left|\mathcal{T}_{b_s}(f_s)\right|^{p_s}\lambda_s\,\d x\right)^{1/p_s}\left(\int \left|\mathcal{B}\right|^{r}\left(\prod_{s=\ell_1+1}^{\ell_2}\lambda_s^{r/p_s}\right)\left(\prod_{s=\ell_2+1}^m\mu_s^{r/p_s}\right)\,\d x\right)^{1/r},
\end{split}
\]where $1/r:=1-\sum_{s=1}^{\ell_1}1/p_s$. It is shown in \cite{Lerner} that
\[
\|\mathcal{T}_{b_s}(f_s)\|_{L^{p_s}(\lambda_s)}\lesssim \left([\mu_s]_{A_{p_s}}[\lambda_s]_{A_{p_s}}\right)^{\max\left(1, \frac{1}{p_s-1}\right)}\|b_s\|_{\text{BMO}(\nu_s)}\|f_s\|_{L^{p_s}(\mu_s)},\quad s=1,\ldots,\ell_1.
\]Observing that $1/r=\sum_{s=\ell_1+1}^m 1/p_s$, one can apply the well known sharp weighted bound of the classical $(m-\ell_1)$-linear sparse operator to obtain
\[
\begin{split}
&\left(\int \left|\mathcal{B}\right|^{r}\left(\prod_{s=\ell_1+1}^{\ell_2}\lambda_s^{r/p_s}\right)\left(\prod_{s=\ell_2+1}^m\mu_s^{r/p_s}\right)\,\d x\right)^{1/r}\\
\lesssim &[\vec{w}_1]_{A_{\vec{r}}}^{\max\left(1,\frac{p'_{\ell_1+1}}{r},\ldots,\frac{p'_m}{r}\right)}\left(\prod_{s=\ell_1+1}^{\ell_2}\left\|\mathcal{A}^1_{\widetilde{\mathcal{S}}}(f_s)\nu_s\right\|_{L^{p_s}(\lambda_s)}\right)\left(\prod_{s=\ell_2+1}^m\|f_s\|_{L^{p_s}(\mu_s)}\right)
\end{split}
\]where $\vec{w}_1:=(\lambda_{\ell_1+1},\ldots,\lambda_{\ell_2},\mu_{\ell_2+1},\ldots,\mu_m)$ and $\vec{r}:=(p_{\ell_1+1},\ldots,p_m)$. Note that 
\[
[\vec{w}_1]_{A_{\vec{r}}}\leq \left(\prod_{s=\ell_1+1}^{\ell_2}[\lambda_s]_{A_{p_s}}^{\frac{r}{p_s}}\right)[\vec{w}]_{A_{\vec{q}}}^{\frac{r}{q}},
\]where $1/q=1/p_{\ell_2+1}+\cdots+1/p_m$, and 
\[
\left\|\mathcal{A}^1_{\widetilde{\mathcal{S}}}(f_s)\nu_s\right\|_{L^{p_s}(\lambda_s)}=\left\|\mathcal{A}^1_{\widetilde{\mathcal{S}}}(f_s)\right\|_{L^{p_s}(\mu_s)}\lesssim [\mu_s]_{A_{p_s}}^{\max\left(1,\frac{1}{p_s-1}\right)}\|f_s\|_{L^{p_s}(\mu_s)},\quad s=\ell_1+1,\ldots,\ell_2.
\]Collecting together the estimates above, one has
\[
\|\mathcal{A}(\bm{f})\|_{L^p(\nu_{\vec\lambda})}\lesssim C'(\vec\mu,\vec\lambda,\vec{p})\left(\prod_{s=1}^{\ell_2}\|b_s\|_{\text{BMO}(\nu_s)}\right)\left(\prod_{s=1}^m\|f_s\|_{L^{p_s}(\mu_s)}\right),
\]where
\[
\begin{split}
C'(\vec\mu,\vec\lambda,\vec{p})=&\left(\prod_{s=1}^{\ell_1}\left([\mu_s]_{A_{p_s}}[\lambda_s]_{A_{p_s}}\right)^{\max\left(1, \frac{1}{p_s-1}\right)}\right)\cdot\\
&\quad\left(\prod_{s=\ell_1+1}^{\ell_2}[\mu_s]_{A_{p_s}}^{\max\left(1,\frac{1}{p_s-1}\right)}[\lambda_s]_{A_{p_s}}^{\frac{1}{p_s}\max\left(r,p'_{\ell_1+1},\ldots,p'_m\right)}\right)[\vec{w}]_{A_{\vec{q}}}^{\frac{1}{q}\max\left(r,p'_{\ell_1+1},\ldots,p'_m\right)}.
\end{split}
\]

The estimate above holds true for the degenerate cases $\ell_1=0$, $\ell_1=\ell_2=0$, or $\ell_1=\ell_2$ as well, therefore, maximizing $C'$ over $p\leq r\leq q$ gives
\[
\begin{split}
C'(\vec\mu,\vec\lambda,\vec{p})\leq &\left(\prod_{s=1}^{\ell_1}\left([\mu_s]_{A_{p_s}}[\lambda_s]_{A_{p_s}}\right)^{\max\left(1, \frac{1}{p_s-1}\right)}\right)\cdot\\
&\quad\left(\prod_{s=\ell_1+1}^{\ell_2}[\mu_s]_{A_{p_s}}^{\max\left(1,\frac{1}{p_s-1}\right)}[\lambda_s]_{A_{p_s}}^{\frac{1}{p_s}\max\left(q,p'_{\ell_1+1},\ldots,p'_m\right)}\right)[\vec{w}]_{A_{\vec{q}}}^{\frac{1}{q}\max\left(q,p'_{\ell_1+1},\ldots,p'_m\right)}\\
\leq &C(\vec\mu,\vec\lambda,\vec{p}),
\end{split}
\]which completes the proof.

\end{proof}

\subsection{Higher order multilinear commutators}\label{Smulti1}

We prove Theorem \ref{multi2} in this subsection, which then immediately implies Corollary \ref{higherordercor}, the two-weight inequality for higher order multilinear commutators. We assume that $k_j\geq 1$, $\forall j$. The case when some $k_j$'s are $0$ is easier but requires extra care of notations, which is why we omit. 

Given any function $h$, define its multiplication operator in the $\beta$-th component by
\[
M^\beta_h:\,(f_1,\ldots,f_m)\mapsto (f_1,\ldots, hf_\beta,\ldots,f_m),\quad \beta\in\{1,\ldots,m\}.
\]Let
\[
F(\vec{z}):=e^{\sum_{j=1}^m\sum_{i=1}^{k_j} b_j^iz_j^i}\widetilde{T}\circ\big[M^{m}_{e^{-\sum_{i=1}^{k_m}b_m^iz_m^i}}\circ\cdots\circ M^{1}_{e^{-\sum_{i=1}^{k_1}b_1^iz_1^i}}\big],
\]then there holds
\[
C^{\vec{k},m}_{\{\vec{b}\}}(\widetilde{T})=\partial_{z_1^1}\cdots\partial_{z_m^{k_m}}F(\vec{0}).
\]The Cauchy integral formula on the polydiscs then implies that
\[
\begin{split}
&\left\|C^{\vec{k},m}_{\{\vec{b}\}}(\widetilde{T})\right\|_{L^{p_1}(\mu_1)\times\cdots\times L^{p_m}(\mu_m)\to L^p(\nu_{\vec\lambda})}\\
\leq&\frac{1}{(2\pi)^{\sum_{j=1}^mk_j}}\oint _{|z_1^1|=\delta_1^1}\cdots\oint_{|z_m^{k_m}|=\delta_m^{k_m}}\|\widetilde{T}\|_{\prod_{j=1}^mL^{p_j}\left(e^{p_j\text{Re}(\sum_{i=1}^{k_j}b_j^iz_j^i)}\mu_j\right)\to L^p\left(e^{p\text{Re}(\sum_{j=1}^m\sum_{i=1}^{k_j}b_j^iz_j^i)}\nu_{\vec\lambda}\right)}\frac{|\d z_1^1|\cdots |\d z_m^{k_m}|}{(\delta_1^1)^2\cdots (\delta_m^{k_m})^2}\\
\leq&\frac{1}{(2\pi)^{\sum_{j=1}^mk_j}}\oint _{|z_1^1|=\delta_1^1}\cdots\oint_{|z_m^{k_m}|=\delta_m^{k_m}}C_{n,m,\vec{p},\widetilde{T}}\left([\vec{w}]_{A_{\vec{p}}},[\vec{v}]_{A_{\vec{p}}}\right)\frac{|\d z_1^1|\cdots |\d z_m^{k_m}|}{(\delta_1^1)^2\cdots(\delta_m^{k_m})^2},
\end{split}
\]where the new multiple weights $\vec{w}, \vec{v}$ are defined as 
\[
w_j:=e^{p_j\text{Re}(\sum_{i=1}^{k_j}b_j^iz_j^i)}\mu_j,\quad v_j:=e^{p_j\text{Re}(\sum_{i=1}^{k_j}b_j^iz_j^i)}\lambda_j,\qquad 1\leq j\leq m,
\]and observe that
\[
\frac{w_j}{v_j}=\frac{\mu_j}{\lambda_j}=\nu_j^{p_j}.
\]
Note that in the first inequality of the display above, one has applied the Minkowski inequality to pass the operator norm inside the integral, which is why the assumption $p\geq 1$ is necessary. Applying Lemma \ref{multilem} below iteratively, one can choose
\[
\delta_j^{i}=\frac{\min\left(\epsilon_{n,m,\vec{p},\vec\mu},\epsilon_{n,m,\vec{p},\vec\lambda}\right)}{p_j\|b_j^i\|_{\text{BMO}}},\quad i=1,\ldots, k_j, \,j=1,\ldots,m,
\]and obtain
\[
C_{n,m,\vec{p},\widetilde{T}}\left([\vec{w}]_{A_{\vec{p}}},[\vec{v}]_{A_{\vec{p}}}\right)\leq C_{n,m,\vec{p},\widetilde{T}}\left(C_{n,m,\vec{p}}[\vec\mu]_{A_{\vec{p}}},C_{n,m,\vec{p}}[\vec\lambda]_{A_{\vec{p}}}\right)=:C'_{n,m,\vec{p},\widetilde{T}}\left([\vec{\mu}]_{A_{\vec{p}}},[\vec\lambda]_{A_{\vec{p}}}\right)
\]by monotonicity of $C_{n,m,\vec{p},\widetilde{T}}$. Then the proof is concluded by the calculation
\[
\begin{split}
&\left\|C^{\vec{k},m}_{\{\vec{b}\}}(\widetilde{T})\right\|_{L^{p_1}(\mu_1)\times\cdots\times L^{p_m}(\mu_m)\to L^p(\nu_{\vec\lambda})}\\
\leq&\frac{1}{\delta_1^1\cdots\delta_m^{k_m}}C'_{n,m,\vec{p},\widetilde{T}}\left([\vec{\mu}]_{A_{\vec{p}}},[\vec\lambda]_{A_{\vec{p}}}\right)\leq C_{n,m,\vec{p},\vec{k},\widetilde{T}}\left([\vec{\mu}]_{A_p},[\vec\lambda]_{A_{\vec{p}}}\right)\prod_{j=1}^m\prod_{i=1}^{k_j}\|b_j^i\|_{\text{BMO}}.
\end{split}
\]

We are left with justifying the following result, which captures the relation between BMO and multilinear $A_{\vec{p}}$ weights. 
\begin{lemma}\label{multilem}
Let $\vec{w}\in A_{\vec{p}}$, where $\vec{p}=(p_1,\ldots,p_m)$, $1/p=1/p_1+\cdots+1/p_m$, $1<p_j<\infty$, and $\vec{b}=(b_1,\ldots,b_m)\in\big(\text{BMO}(\mathbb{R}^n)\big)^m$. Define
\[
v_j:=e^{\text{Re}(b_jz_j)}w_j,\quad j=1,\ldots, m.
\]Then there are constants $\epsilon_{n,m,\vec{p},\vec{w}}, C_{n,m,\vec{p}}>0$ so that
\[
[\vec{v}]_{A_{\vec{p}}}\leq C_{n,m,\vec{p}}[\vec{w}]_{A_{\vec{p}}}
\]whenever $\vec{z}\in\mathbb{C}^m$ satisfies
\[
|z_j|\leq\frac{\epsilon_{n,m,\vec{p},\vec{w}}}{\|b_j\|_{\text{BMO}}},\quad j=1,\ldots,m.
\]
\end{lemma}

\begin{proof}
Fix a cube $Q$ and recall $\nu_{\vec{v}}=\prod_{j=1}^m v_j^{p/p_j}$, $\nu_{\vec{w}}=\prod_{j=1}^m w_j^{p/p_j}$. One has by H\"older's inequality
\[
\begin{split}
[\vec{v}]_{A_{\vec{p}}}^{1/p}=&\left(\fint_Q \nu_{\vec{v}}\right)^{1/p}\prod_{j=1}^m\left(\fint_Q v_j^{1-p'_j}\right)^{1/p'_j}\\
=&\left(\fint_Q e^{\sum_{j=1}^m\frac{p}{p_j}\text{Re}(b_jz_j)} \nu_{\vec{w}}\right)^{1/p}\prod_{j=1}^m\left(\fint_Q e^{(1-p'_j)\text{Re}(b_jz_j)}w_j^{1-p'_j}\right)^{1/p'_j}\\
\leq &\left(\fint_Q e^{\sum_{j=1}^m\frac{pq'}{p_j}\text{Re}(b_jz_j)}\right)^{1/pq'}\left(\fint_Q \nu_{\vec{w}}^q\right)^{1/pq}\prod_{j=1}^m \left(\fint_Q e^{q'(1-p'_j)\text{Re}(b_jz_j)}\right)^{1/p'_jq'}\left(\fint_Q w_j^{q(1-p'_j)}\right)^{1/p'_jq}
\end{split}
\]where $1<q<\infty$ is constant to be determined in the following. Recall that if $\vec{w}\in A_{\vec{p}}$, then $\nu_{\vec{w}}\in A_{mp}$ and $w_j^{1-p'_j}\in A_{mp'_j}$, $\forall j$. Therefore, there exists $q_0=q_0\left([\vec{w}]_{A_{\vec{p}}}\right)$ such that for all $0<q<q_0$,
\[
\left(\fint_Q \nu_{\vec{w}}^q\right)^{1/q}\lesssim \fint_Q \nu_{\vec{w}},\quad \left(\fint_Q w_j^{q(1-p'_j)}\right)^{1/q}\lesssim \fint_Q w_j^{1-p'_j},\,\forall j.
\]
Taking such a $q$, one then has
\[
\left(\fint_Q \nu_{\vec{w}}^q\right)^{1/pq}\prod_{j=1}^m\left(\fint_Q w_j^{q(1-p'_j)}\right)^{1/p'_jq}\lesssim [\vec{w}]_{A_{\vec{p}}}^{1/p}.
\]

On the other hand, observe that
\[
\begin{split}
&\quad\left(\fint_Q e^{\sum_{j=1}^m\frac{pq'}{p_j}\text{Re}(b_jz_j)}\right)^{1/pq'}\prod_{j=1}^m \left(\fint_Q e^{q'(1-p'_j)\text{Re}(b_jz_j)}\right)^{1/p'_jq'}\\
&=\left(\fint_Q e^{\sum_{j=1}^m\frac{pq'}{p_j}\left(\text{Re}(b_jz_j)-\left\langle \text{Re}(b_jz_j)\right\rangle_Q\right)}\right)^{1/pq'}\prod_{j=1}^m \left(\fint_Q e^{q'(1-p'_j)\left(\text{Re}(b_jz_j)-\left\langle \text{Re}(b_jz_j)\right\rangle_Q\right)}\right)^{1/p'_jq'}\\
&=: A^{1/pq'}\prod_{j=1}^m B_j^{1/p'_jq'}.
\end{split}
\]According to John-Nirenberg inequality, there exists constant $C_n>0$ so that 
\[
\fint_Q e^{\left|b-\langle b\rangle_Q\right|}\leq 2
\]whenever $\|b\|_{\text{BMO}}\leq C_n$. Choose $z\in\mathbb{C}^m$ satisfying
\[
|z_j|\leq \frac{C_n\min\left(\frac{1}{1-p'_j}, \frac{p_j}{p}\right)}{q'm\|b_j\|_{\text{BMO}}},\quad j=1,\ldots,m,
\]then there hold $A\leq 2$, $B_j\leq 2$, $\forall j$, hence
\[
[\vec{v}]_{A_{\vec{p}}}\lesssim 2^{mp/q'}[\vec{w}]_{A_{\vec{p}}}\leq C_{n,m,\vec{p}}[\vec{w}]_{A_{\vec{p}}}.
\]

\end{proof}

\begin{remark}
From the proof, it is easy to see that if for some $I\subset\{1,\ldots,m\}$, $b_s=0$, $\forall s\in I$, which corresponds to $v_s=w_s$, $\forall s\in I$, then the same estimate still holds true whenever $|z_j|$ is controlled for $j\notin I$. This observation allows us to iterate Lemma \ref{multilem} to complete the proof of Theorem \ref{multi2} above.
\end{remark}

\subsection{Lower bound of multilinear commutators}
We prove Theorem \ref{lower bound} in this subsection. It suffices to prove the implication ``(3)$\implies$(1)'', as ``(1)$\implies$(2)'' is a special case of Theorem \ref{basethm} and ``(2)$\implies$(3)'' is obvious. Without loss of generality, assume $\beta = 1$ and $[b, P_{\bm{\epsilon}}^{\bm{\alpha}}]_1: L^{p_1}(w_1) \times \cdots \times L^{p_m}(w_m) \rightarrow L^p(\nu_{\vec{w}})$ is bounded, where $P_{\bm{\epsilon}}^{\bm{\alpha}}$ is a Haar multiplier with respect to dyadic grid $\mathcal{D}$ such that $\alpha_j\neq \vec{1}$ for some $j\in\{2,\ldots,m\}$, and $\{|\epsilon_I|\}$ are bounded from below uniformly. 

Fix a cube $J\in\mathcal{D}$. For $j\in\{1,\ldots,m\}$, let $$f_j = 
     \begin{cases}
       \abs{J}^{1/2}h_{J}^{\alpha_j}, &\text{if }\alpha_j \neq \vec{1},\\
       \chi_{J}, \;\;\; & \text{if }\alpha_j = \vec{1}.\\
            \end{cases}$$
Then there holds $\La f_j, h_j^{\alpha_j}\Ra= |J|^{1/2}$. Moreover,
\[
\La bf_1, h_J^{\alpha_1}\Ra=\La b\Ra_{J}|J|^{1/2}.
\]Hence, the fact that $\alpha_j\neq\vec{1}$ for some $j\in\{2,\ldots,m\}$ implies that
\[
\left|[b,P_{\bm{\epsilon}}^{\bm{\alpha}}]_1(\vec{f})\right|=\left|\epsilon_J bh_J^{\alpha_{m+1}}|J|^{1/2}-\epsilon_J\La b\Ra_{J}|J|^{1/2}h_J^{\alpha_{m+1}}\right|\gtrsim\left|b-\La b\Ra_J\right|\chi_J.
\]

Therefore,
\begin{eqnarray*}
\int_J \left|b - \La b \Ra_J \right|^{\frac{1}{m}} \,\d x 
&=& \int_J \left|b - \La b \Ra_J \right|^{\frac{1}{m}} \nu_{\vec{w}}^{\frac{1}{mp}} \nu_{\vec{w}}^{-\frac{1}{mp}}\, \d x\\
&\leq& \left( \int_J \left|b - \La b \Ra_J \right|^p \nu_{\vec{w}}\,\d x\right)^{\frac{1}{mp}} \left( \int_J \nu_{\vec{w}}^{-\frac{(mp)'}{mp}}\,\d x\right)^{\frac{1}{(mp)'}}\\
&\lesssim& \left\Vert [b, P_{\bm{\epsilon}}^{\bm{\alpha}}]_1(\bm{f}) \right\Vert_{L^p{(\nu_{\vec{w}})}}^{1/m} \left( \int_J \nu_{\vec{w}}^{\frac{1}{1-mp}}\,\d x\right)^{\frac{mp-1}{mp}},
\end{eqnarray*}which by the boundedness of the commutator is less than
\begin{eqnarray*}
&\lesssim& \left(\prod_{j=1}^m \left\Vert f_j \right\Vert_{L^{p_j}{(w_j)}} \right)^{1/m}\left( \int_J \nu_{\vec{w}}^{\frac{1}{1-mp}}\,\d x\right)^{\frac{mp-1}{mp}}\\
&=& \left( \int_J \nu_{\vec{w}}^{{\frac{1}{1-mp}}}\,\d x\right)^{\frac{mp-1}{mp}} \left(\prod_{j=1}^m \left(\int_{J} w_j \,\d x \right)^{1/p_j} \right)^{1/m},
\end{eqnarray*}where the last step above follows by observing $|f_j|=\chi_J$. This further implies
\begin{equation} \label{eq:BMO}\int_J \left|b - \La b \Ra_J \right|^{\frac{1}{m}} \,\d x \lesssim \left( \int_{J} \nu_{\vec{w}}^{{\frac{1}{1-mp}}}\,\d x\right)^{\frac{mp-1}{mp}} \prod_{j=1}^m \left(\int_{J} w_j\,\d x \right)^{\frac{1}{mp_j}}.
\end{equation}
Let $w_j':= w_j^{1-p_j'}, \vec{w}': = (w_1', \ldots, w_m')$, and $\vec{p}':= (p_1', \ldots, p_m')$. Since $ w_j'\in A_{p_j'},$ $\vec{w}'$ satisfies the multilinear $A_{\vec{p}'}$ condition
\begin{equation}
\label{eq:AP'} \sup_Q\left(\fint_{Q} \nu_{\vec{w}'}\right) \prod_{j=1}^m \left( \fint_{Q} (w_j')^{1-p_j}\right)^{\frac{p/(mp-1)}{p_j}} = [\vec{w}']_{A_{\vec{p}'}} < \infty,
\end{equation}observing that $\sum_{j=1}^m \frac{1}{p_j'} = m - \sum_{j=1}^m \frac{1}{p_j}= \frac{mp - 1}{p}$.

Moreover, as $\displaystyle \nu_{\vec{w}'} = \nu_{\vec{w}}^{\frac{1}{1-mp}}$ and $(w_j')^{1-p_j} = w_j$, $\forall j$, \eqref{eq:AP'} implies that
\[
\left(\int_{J} \nu_{\vec{w}}^{\frac{1}{1-mp}} \right) \prod_{j=1}^m \left(\int_{J} w_j \right)^{\frac{p/(mp-1)}{p_j}} \leq[\vec{w}']_{A_{\vec{p}'}} \abs{J}^{1+\frac{p}{mp-1}(\frac{1}{p_1} + \cdots + \frac{1}{p_m})}=[\vec{w}']_{A_{\vec{p}'}} \abs{J}^{\frac{mp}{mp-1}}.
\]
Consequently, one deduces from \eqref{eq:BMO} that
\[
\int_J|b-\langle b\rangle_J|^{\frac{1}{m}}\,\d x\lesssim [\vec{w}']_{A_{\vec{p}'}}^{\frac{mp-1}{mp}} \abs{J},
\]which implies
$$ \left(\fint_J \left|b - \La b \Ra_J \right|^{\frac{1}{m}} \,\d x \right)^{m} \lesssim [\vec{w}']_{A_{\vec{p}'}}^{\frac{mp-1}{p}}.$$
Taking supremum over $J\in \D$ on the LHS above shows that $b \in \text{BMO}_{\mathcal{D}}$, thus the proof is complete observing that $\text{BMO}=\bigcap_{\mathcal{D}}\text{BMO}_{\mathcal{D}}$.

\begin{remark}
In the proof above, we in fact only used the property $\vec{w}'\in A_{\vec{p}'}$, which is  slightly weaker than the assumption that $w_j\in A_{p_j}$, $\forall j=1,\ldots,m$.
\end{remark}
\begin{remark}
Statement (3) of Theorem \ref{lower bound} can be further weakened, with ``For all dyadic grid $\mathcal{D}$...'' replaced by ``For $\mathcal{D}_1,\ldots,\mathcal{D}_{n+1}$ such that $\text{BMO}=\bigcap_{j=1}^{n+1}\text{BMO}_{\mathcal{D}_j}$...''. It is proved in \cite{Mei} that in order to recover the continuous BMO, it suffices to take intersection of $n+1$ shifted dyadic BMO. Therefore the weakened version of (3) still implies (1).
\end{remark}



\appendix
\section{Proof of Theorem \ref{basethm} for dyadic operators}\label{SecApp}
Fix a dyadic grid $\mathcal{D}$. According to Lemma \ref{SparseBloom}, Theorem \ref{basethm} in the dyadic case follows from the following sparse domination of multilinear commutators with Haar multipliers or paraproducts. 
\begin{proposition}\label{CommDomDya} Let $T$ be either a Haar multiplier $P_{\bm{\epsilon}}^{\bm{\alpha}}$ or a paraproduct $\pi_{g,\bm{\epsilon}}^{\bm{\alpha}}$ (with symbol function $g\in\text{BMO}_{\mathcal{D}}$) as defined in Subsection \ref{SubSecDef}. Given locally integrable functions $\bm{b}=(b_{i_1},\ldots,b_{i_\ell})$ on $\mathbb{R}^n$, there exists a constant $C=C(n,T)$ so that for any bounded functions $\bm{f}=(f_1,\ldots,f_m)$ with compact support, there exists a sparse collection $\mathcal{S}=\mathcal{S}(T,\bm{f},\bm{b})\subset\mathcal{D}$ such that
\[
\left|[b_{i_1},\cdots,[b_{i_\ell},T]_{i_\ell}\cdots]_{i_1}(\bm{f})\right|\leq C\left(\sum_{\vec{\gamma}\in \{1,2\}^\ell}\mathcal{A}^{\vec{\gamma}}_{\mathcal{S},\bm{b}}(\bm{f})\right),\quad \text{a.e.}
\]where the sparse operators $\mathcal{A}^{\vec{\gamma}}_{\mathcal{S},\bm{b}}$ are defined as in Proposition \ref{CommDomCZO}.
\end{proposition}

We omit the proof of Proposition \ref{CommDomDya}, as it proceeds in the same way as Proposition \ref{CommDomCZO} with minimal modification once we verify that the dyadic maximal truncated operators in this case map $L^1\times\cdots\times L^1$ into $L^{\frac1m,\infty}$.
\begin{lemma}\label{LemTrunc}
Let $T$ be either a Haar multiplier $P_{\bm{\epsilon}}^{\bm{\alpha}}$ or a paraproduct $\pi_{g,\bm{\epsilon}}^{\bm{\alpha}}$. Then
\[
\|T_\sharp\|_{L^{\frac1m,\infty}}\lesssim \prod_{j=1}^m\|f_j\|_{L^1}
\]where $T_\sharp f$ is the maximal truncated operator given by
\[
T_\sharp \bm{f}:=\sup_{J\in\mathcal{D}}|T^J \bm{f}|\] where
\[T^J \bm{f}:=
\begin{cases}
\sum_{\substack{I\in\mathcal{D}\\ I\supsetneq J}}\epsilon_I\La f_1,h_I^{\alpha_1}\Ra \cdots \La f_m,h_I^{\alpha_m}\Ra h_I^{\alpha_{m+1}}|I|^{-(m-1)/2} &\text{if }T=P_{\bm{\epsilon}}^{\bm{\alpha}},\\
\sum_{\substack{I\in\mathcal{D}\\ I\supsetneq J}} \epsilon_I\La g, h_I^{\alpha_1} \Ra\left(\prod_{j=1}^m \La f_j, h_I^{\alpha_{j+1}}\Ra\right) h_I^{\alpha_{m+2}}|I|^{-m/2} &\text{if }T=\pi_{g,\bm{\epsilon}}^{\bm{\alpha}}.
\end{cases}
\]
\end{lemma}
\begin{proof} It is easy to see that $T_\sharp$ is multi-sublinear, and that $T_\sharp(\bm{f})$ is supported on $I\in\D$ if $f_j = h_I$ for some $j$ where $h_I$ denotes any cancellative Haar function associated to $I$. Hence by Lemma 3.5 in \cite{Kunwar}, it suffices to prove
\begin{equation}\label{EqnRed}
\left\Vert T_\sharp(\bm{f}) \right\Vert_{L^{p}} \lesssim \prod_{j=1}^m\left\Vert f_j \right\Vert_{L^{p_j}},
\end{equation}for some $1 <p_j, p < \infty$ with $\sum_{j=1}^m\frac{1}{p_j} = \frac{1}{p}$. (Note that Lemma 3.5 of \cite{Kunwar} deals only with the case $\mathbb{R}$, however it is easy to see that the argument extends to $\mathbb{R}^n$.)

We claim that
\begin{equation}\label{EqnClaim}
T_\sharp(\bm{f})(x) \leq M(T(\bm{f}))(x) \quad \text{ for every } x \in \mathbb{R}^n,
\end{equation}where $M$ denotes the Hardy-Littlewood maximal function. It is easy to see that (\ref{EqnClaim}) implies (\ref{EqnRed}) immediately since operator $T$ is bounded. To see (\ref{EqnClaim}), fix a $J\in\mathcal{D}$ and a point $x\in\mathbb{R}^n$. Let $J'$ be the smallest dyadic cube that contains both $x$ and $J$ properly. Note that if there is no such $J'$ existing, then $T^J(\bm{f})(x)=0$ which trivially implies (\ref{EqnClaim}). Denote also by $J''$ the child of $J'$ that contains $x$. Without loss of generality, we assume that $\alpha_{m+1}\neq\vec{1}$ if $T=P_{\bm{\epsilon}}^{\bm{\alpha}}$ and $\alpha_{m+2}\neq\vec{1}$ if $T=\pi_{g,\bm{\epsilon}}^{\bm{\alpha}}$. This is always possible by duality, since the desired estimate (\ref{EqnRed}) is inside the Banach range hence symmetric and we have assumed in the definition that some of the Haar functions in both cases need to be cancellative. Therefore, take $T=P_{\bm{\epsilon}}^{\bm{\alpha}}$ as an example, one has
\[
\begin{split}
T^J(\bm{f})(x)&=\sum_{\substack{I\in\mathcal{D}\\ I\supset J'}}\epsilon_I\La f_1,h_I^{\alpha_1}\Ra \cdots \La f_m,h_I^{\alpha_m}\Ra h_I^{\alpha_{m+1}}(x)|I|^{-(m-1)/2}\\
&=\left(\sum_{\substack{I\in\mathcal{D}\\ I\supsetneq J''}}\epsilon_I\La f_1,h_I^{\alpha_1}\Ra \cdots \La f_m,h_I^{\alpha_m}\Ra h_I^{\alpha_{m+1}}(x)|I|^{-(m-1)/2}\right)\chi_{J''}(x)=\La T(\bm{f})\Ra_{J''}\chi_{J''}(x),
\end{split}
\]which concludes the proof of (\ref{EqnClaim}). The case with $T=\pi_{g,\bm{\epsilon}}^{\bm{\alpha}}$ follows similarly.
\end{proof}
\begin{remark}
Another consequence of Lemma \ref{LemTrunc} is that for any given $\bm{f}$, dyadic operator $T=P_{\bm{\epsilon}}^{\bm{\alpha}}$ or $\pi_{g,\bm{\epsilon}}^{\bm{\alpha}}$ has a pointwise sparse bound
\[
|T(\bm{f})|\leq C\mathcal{A}_{\mathcal{S}}(\bm{f}):=\sum_{I\in\mathcal{S}}\left(\prod_{j=1}^m \La |f_j|\Ra_I\right)\chi_I,
\]which entails boundedness properties of $T$ such as quantitative multilinear weighted estimates via standard deduction.
\end{remark}

\section*{Acknowledgments} 
The authors would like to thank Michael Lacey for having made this collaboration possible. The authors are also grateful to Rodolfo Torres for kindly pointing out several inaccuracies of references in an earlier version of the paper. The second author was in residence at the Mathematical Sciences Research Institute in Berkeley, California, during the Spring 2017 semester when this work was carried out, and was supported by the National Science Foundation under Grant No. DMS-1440140.

\bibliography{HigherOrdComm}{}
\bibliographystyle{amsplain}
\end{document}